\documentclass[11pt]{article}

\usepackage[T1]{fontenc}
\usepackage{lmodern}
\usepackage{amsmath,amssymb,amsthm,mathtools}
\usepackage[margin=1.02in]{geometry}
\usepackage{booktabs,tabularx,array}
\usepackage{microtype}
\usepackage{xcolor}
\usepackage[colorlinks=true,linkcolor=blue!50!black,
  citecolor=blue!50!black,urlcolor=blue!50!black]{hyperref}

\hypersetup{
  pdftitle={A universal leading-residue formula for Witten zeta functions},
  pdfauthor={Jonas Matuzas},
  pdfsubject={Witten zeta functions and the Macdonald--Mehta integral},
  pdfkeywords={Witten zeta function, Macdonald--Mehta integral, Coxeter group,
    representation growth}}

\newtheorem{theorem}{Theorem}[section]
\newtheorem{proposition}[theorem]{Proposition}
\newtheorem{corollary}[theorem]{Corollary}
\newtheorem{lemma}[theorem]{Lemma}
\newtheorem{remark}[theorem]{Remark}

\newcommand{\R}{\mathbb{R}}
\newcommand{\C}{\mathbb{C}}
\newcommand{\N}{\mathbb{N}}
\newcommand{\Qbar}{\overline{\mathbb{Q}}}
\newcommand{\Res}{\operatorname*{Res}}
\newcommand{\vol}{\operatorname{vol}}
\newcommand{\dd}{\,\mathrm{d}}

\title{A universal leading-residue formula\\
for Witten zeta functions}
\author{Jonas Matuzas\\
\href{mailto:jonas.matuzas@gmail.com}{\nolinkurl{jonas.matuzas@gmail.com}}}
\date{}

\begin{document}
\maketitle

\begin{abstract}
Let \(\Phi\) be an irreducible crystallographic root system of rank \(r\),
with Coxeter number \(h\), Weyl group \(W\), Cartan matrix \(C_\Phi\), and
invariant degrees \(2=d_1\leq\cdots\leq d_r=h\).  We prove that Au's
normalized Witten zeta function has a simple pole at \(2/h\) and evaluate
its residue as
\[
 \frac{2(2\pi)^{r/2}\sqrt{\det C_\Phi}}{h|W|}
 \frac{\prod_{i<r}\Gamma(1-d_i/h)}{\Gamma(1-1/h)^r}.
\]
The central step is Proposition~\ref{prop:critical-sphere}: it evaluates the
critical chamber integral in gamma values from the boundary pole of the
Macdonald--Mehta--Opdam identity.  Sections~\ref{sec:faces} and
\ref{sec:transfer} are preparatory; they pass from the dominant-weight
lattice to that convergent integral.  This proves Au's conjecture on
algebraic multiples of products of gamma values at rational arguments,
including his \(A_4\) prediction.
\end{abstract}

\noindent\textbf{2020 Mathematics Subject Classification.}
Primary 11M41; Secondary 11P21, 20F55, 22E46, 33C67.

\noindent\textbf{Keywords.}
Witten zeta function; representation growth; Macdonald--Mehta integral;
finite Coxeter group; root system; lattice-point asymptotics.

\paragraph{Generative-AI disclosure and author responsibility.}
This work was produced using \emph{OpenAI's ChatGPT 5.6 Pro}.  The author
directed and audited the work throughout.  Jonas Matuzas takes full
responsibility for the mathematics and the final text.

\section{Introduction}

For a complex simple Lie algebra, or equivalently its simply connected
compact form, the Witten zeta function is
\begin{equation}\label{eq:witten-definition}
 \zeta_\Phi(s)=\sum_{\lambda\in P_+}(\dim V_\lambda)^{-s},
\end{equation}
where \(P_+\) is the set of dominant integral weights.  Witten introduced
these series in connection with two-dimensional gauge theory
\cite{Witten1991}; they are also representation zeta functions.  The theory
of zeta functions of root systems provides meromorphic continuation and
many special-value identities \cite{MatsumotoTsumura2006,
KomoriMatsumotoTsumura2023}.  General polynomial Dirichlet-series results
give continuation and leading-pole information in a much wider setting
\cite[Theorems~2--4]{Essouabri1997}.  Larsen--Lubotzky and
H\"as\"a--Stasinski proved that the abscissa in the simple case is
\(2/h\) \cite[Theorem~5.1]{LarsenLubotzky2008}
\cite[Corollary~2.4 and Lemma~2.7]{HasaStasinski2019}.

Au computed the leading residues in ranks two and three
\cite[Tables~1--2]{Au2024}.  His Conjecture~1.3 predicts, for every root
system, that the residue of the normalized function at its abscissa is an
element of \(\Qbar\) times a product of gamma values at rational arguments;
it also records a numerical prediction in type \(A_4\).  Au later reiterated
the expectation that general leading residues admit gamma-function
expressions \cite[pp.~38--40]{Au2026Oberwolfach}.  Theorem~\ref{thm:main}
gives a uniform closed formula and proves that prediction.

The gamma product itself comes from established work.  Macdonald conjectured
the Gaussian discriminant integral \cite{Macdonald1982}; Opdam proved it
\cite{Opdam1989,Opdam1993}, and Etingof supplied a uniform proof for all
finite Coxeter groups \cite{Etingof2010}.  The first singular parameter
\(1/h\) and the associated normalized distributions also occur in rational
Cherednik theory \cite[Remark~4.11]{EtingofStoica2009}
\cite[Theorem~3.6 and Proposition~3.9]{EtingofSupports2012}.  Likewise,
local integrability of hyperplane products admits a general formulation in
terms of arrangement flats and real log-canonical thresholds
\cite[Corollary~4.4]{KostaWindisch2026}.  The contribution here is the passage
from the dominant-weight lattice to the convergent spherical integral at the
critical exponent,
its endpoint evaluation, and the cancellation that reduces all metric data
to \(\sqrt{\det C_\Phi}\).
To the best of our knowledge, no previous work gives this universal
leading-residue formula in the present normalization.

\paragraph{Main idea of the proof.}
Weyl's dimension formula turns the normalized zeta function into a lattice
sum of a homogeneous product \(P\) of positive-coroot linear forms.  Every
proper coordinate face is subcritical because each proper parabolic Coxeter
number is smaller than \(h\).  Coordinatewise monotonicity then shows that
the lattice sum and its chamber integral differ by a bounded function at
\(s=2/h\).  Polar coordinates isolate a convergent angular integral.  Its
value is the quotient of the two logarithmic residues in the radial
factorization of the Macdonald--Mehta integral at \(k=-1/h\).  Finally,
root-length counts and the Cartan--Gram determinant identity cancel the
chosen Euclidean metric.

\section{Conventions and main theorem}
\label{sec:conventions}

Let \(\Phi\) be an irreducible crystallographic root system in a real
Euclidean space, let \(W\) be its Weyl group, and put
\(\Psi=\Phi^\vee\).  Let \(\Psi^+\) be a positive
coroot system with simple coroots \(\beta_1,\ldots,\beta_r\).  The
fundamental weights \(\omega_1,\ldots,\omega_r\) satisfy
\((\omega_i,\beta_j)=\delta_{ij}\).

We normalize the Cartan matrix of \(\Psi\) by
\[
 (C_\Psi)_{ij}=\langle\beta_j,\beta_i^\vee\rangle
 =\frac{2(\beta_i,\beta_j)}{\|\beta_i\|^2}.
\]
Thus \(C_\Psi=C_\Phi^{\mathsf T}\).  We write
\(G=((\beta_i,\beta_j))\) for the simple-coroot Gram matrix.

The invariant degrees of the irreducible Weyl group are
\(2=d_1\leq\cdots\leq d_r=h\).  The top degree \(h\) has
multiplicity one.  The exponents are \(m_i=d_i-1\), and
\(N=|\Psi^+|=rh/2\).

For \(r=1\), surface measure on \(S^0=\{-1,1\}\) is counting
measure, so each point has mass one.
The classical rank ranges are \(A_r\) for \(r\geq1\), \(B_r,C_r\) for
\(r\geq2\), and \(D_r\) for \(r\geq4\).  We treat
\(B_1=C_1=A_1\) and \(D_3=A_3\) as accidental isomorphisms.

Let \(\rho=\sum_i\omega_i\).  Au's normalization is
\begin{equation}\label{eq:normalization}
 \xi_\Phi(s)=K_\Phi^{-s}\zeta_\Phi(s),
 \qquad
 K_\Phi=\prod_{\beta\in\Psi^+}(\rho,\beta).
\end{equation}
If \(\operatorname{ht}(\beta)\) denotes coroot height, the standard
height--exponent identity says
\[
 \sum_{\beta\in\Psi^+}t^{\operatorname{ht}(\beta)}
 =\sum_{i=1}^r(t+t^2+\cdots+t^{m_i});
\]
see, for example, the finite-reflection-group discussion in
\cite{Steinberg1959,Humphreys1990}.  Comparing the multiplicities of each
height gives
\begin{equation}\label{eq:K-degrees}
 K_\Phi=\prod_{j\geq1}j^{\#\{i:m_i\geq j\}}
       =\prod_{i=1}^r m_i!
       =\prod_{i=1}^r(d_i-1)!.
\end{equation}

\begin{theorem}\label{thm:main}
Let \(\Phi\) be an irreducible crystallographic root system of rank \(r\).
Then \(\xi_\Phi(s)\) has a simple pole at \(s_0=2/h\), and
\begin{equation}\label{eq:main}
  \Res_{s=2/h}\xi_\Phi(s)=
 \frac{2(2\pi)^{r/2}\sqrt{\det C_\Phi}}{h|W|}
 \frac{\displaystyle\prod_{i=1}^{r-1}
       \Gamma\!\left(1-\frac{d_i}{h}\right)}
      {\displaystyle\Gamma\!\left(1-\frac1h\right)^r}.
\end{equation}
Consequently,
\begin{equation}\label{eq:ordinary-residue}
 \Res_{s=2/h}\zeta_\Phi(s)
 =\left(\prod_{i=1}^r(d_i-1)!\right)^{2/h}
  \Res_{s=2/h}\xi_\Phi(s).
\end{equation}
\end{theorem}

\section{Convergence on the boundary faces}
\label{sec:faces}

Write a dominant weight as \(\lambda=\sum_i n_i\omega_i\), put
\(m_i=n_i+1\), and set \(x_m=\sum_i m_i\omega_i=\lambda+\rho\).  If
\(\beta=\sum_i b_i(\beta)\beta_i\), Weyl's dimension formula gives
\[
 \dim V_\lambda
 =\frac{\prod_{\beta\in\Psi^+}(x_m,\beta)}{K_\Phi}
 =\frac{P(m)}{K_\Phi},
\]
where
\begin{equation}\label{eq:P}
 P(x)=\prod_{\beta\in\Psi^+}
       \left(\sum_{i=1}^r b_i(\beta)x_i\right).
\end{equation}
Thus
\begin{equation}\label{eq:xi-lattice}
 \xi_\Phi(s)=\sum_{m\in\N^r}P(m)^{-s}.
\end{equation}
Here and below \(\N=\{1,2,\ldots\}\).  The polynomial \(P\) is homogeneous
of degree \(N=rh/2\) and is nondecreasing in every coordinate.

\begin{lemma}\label{lem:faces}
Fix at least one coordinate in \eqref{eq:xi-lattice} equal to \(1\).
The resulting face sum converges at \(s=s_0=2/h\), including when all
coordinates are fixed.
\end{lemma}

\begin{proof}
If all coordinates are fixed, the face is one term.  Otherwise let \(F\)
be the set of free indices, with \(1\leq|F|<r\).  For every nonempty
\(S\subseteq F\), let \(N(S)\) be the number of positive coroots whose
simple-coroot support meets \(S\), and put \(J=\{1,\ldots,r\}\setminus S\).
The roots not counted by \(N(S)\) form the parabolic system \(\Psi_J\).
Write its irreducible components as \(\Psi_a\), with ranks \(r_a\) and
Coxeter numbers \(h_a\).  Every component is proper, hence \(h_a<h\);
this is the strict proper-Levi inequality of
\cite[Lemma~2.7]{HasaStasinski2019}.  Since
\(|\Psi_a^+|=r_ah_a/2\) and \(\sum_a r_a=r-|S|\),
\begin{equation}\label{eq:parabolic-strict}
 \begin{aligned}
 s_0N(S)-|S|
 &=\frac2h\left(N-|\Psi_J^+|\right)-|S| \\
 &=\sum_a r_a\left(1-\frac{h_a}{h}\right)>0.
 \end{aligned}
\end{equation}

Decompose each free variable into dyadic intervals
\(2^{k_i}\leq m_i<2^{k_i+1}\).  A root form whose support misses \(F\)
is constant on the face.  Every other root form is comparable, with constants
independent of the box, to
\(2^{\max_{i\in\operatorname{supp}(\beta)\cap F}k_i}\).  For
\(S_\ell=\{i\in F:k_i\geq\ell\}\), the negative of the base-two exponent
in the box contribution at \(s_0\) is
\[
 \sum_{\ell\geq1}\bigl(s_0N(S_\ell)-|S_\ell|\bigr).
\]
There are finitely many nonempty subsets of \(F\), so
\eqref{eq:parabolic-strict} is bounded below by a positive constant
\(\delta\) whenever \(S_\ell\neq\varnothing\).  The last display is at
least \(\delta\max_i k_i\).  Only polynomially many boxes have a given
maximum, and the face sum converges.
\end{proof}

\section{Comparison of the lattice sum with a chamber integral}
\label{sec:transfer}

\begin{proposition}\label{prop:transfer}
Let
\[
 J(s)=\int_{[1,\infty)^r}P(x)^{-s}\dd x,
 \qquad \Re s>s_0.
\]
Then, for real \(s>s_0\),
\begin{equation}\label{eq:bounded-transfer}
 0\leq\xi_\Phi(s)-J(s)\leq B(s_0)<\infty,
\end{equation}
where \(B(s_0)\) is a finite sum of proper-face series.  Moreover,
\begin{equation}\label{eq:residue-angular-raw}
 \Res_{s=s_0}\xi_\Phi(s)
 =\frac1{ND}\int_{\mathcal S}
       \Delta_\Psi(\theta)^{-2/h}\dd\sigma(\theta),
\end{equation}
where
\[
 D=(\det G)^{-1/2},\qquad
 \mathcal S=\mathcal C\cap S^{r-1},\qquad
 \Delta_\Psi(v)=\prod_{\beta\in\Psi^+}(v,\beta),
\]
and \(\mathcal C\) is the dominant chamber.
\end{proposition}

\begin{proof}
Put \(f_s=P^{-s}\).  For \(m\in\N^r\), the forward cubes
\(Q_m^+=\prod_i[m_i,m_i+1)\) partition \([1,\infty)^r\), and
coordinatewise monotonicity gives
\begin{equation}\label{eq:cube-lower}
 J(s)\leq\sum_{m\in\N^r}f_s(m)=\xi_\Phi(s).
\end{equation}
For \(m_i\geq2\) for every \(i\), the backward cubes
\(Q_m^-=\prod_i[m_i-1,m_i)\) also partition \([1,\infty)^r\), and
\(f_s(x)\geq f_s(m)\) on \(Q_m^-\).  Therefore
\begin{equation}\label{eq:cube-upper}
 \sum_{\substack{m\in\N^r\\m_i\geq2\ (1\leq i\leq r)}}f_s(m)
 \leq J(s).
\end{equation}
The omitted points lie in the union of the coordinate faces \(m_j=1\).
Lemma~\ref{lem:faces}, together with \(P(m)\geq1\), proves
\eqref{eq:bounded-transfer}.  Hence \(\xi_\Phi\) and \(J\) have the same
leading coefficient at \(s_0\).

The fundamental-weight basis has Gram matrix \(G^{-1}\), so its Euclidean
covolume is \(D=(\det G)^{-1/2}\).  Under
\(x\mapsto v=\sum_i x_i\omega_i\), one has \(\dd v=D\dd x\), and
\begin{equation}\label{eq:J-chamber}
 J(s)=\frac1D
 \int_{\substack{v\in\mathcal C\\(v,\beta_i)\geq1}}
 \Delta_\Psi(v)^{-s}\dd v.
\end{equation}
In polar coordinates \(v=\rho\theta\), put
\[
 a(\theta)=\min_i(\theta,\beta_i),
 \qquad R(\theta)=a(\theta)^{-1}.
\]
The radial lower limit in \eqref{eq:J-chamber} is exactly \(R(\theta)\),
so homogeneity gives
\begin{equation}\label{eq:polar-J}
 J(s)=\frac1{D(Ns-r)}
 \int_{\mathcal S}\Delta_\Psi(\theta)^{-s}
 R(\theta)^{r-Ns}\dd\sigma(\theta).
\end{equation}
Because each positive coroot is a nonnegative integral combination of the
simple coroots, \(\Delta_\Psi(\theta)\geq a(\theta)^N\).  Thus, for real
\(s\geq s_0\),
\begin{equation}\label{eq:domination}
 \Delta_\Psi(\theta)^{-s}R(\theta)^{r-Ns}
 =\Delta_\Psi(\theta)^{-s_0}
 \left(\frac{a(\theta)^N}{\Delta_\Psi(\theta)}\right)^{s-s_0}
 \leq\Delta_\Psi(\theta)^{-s_0}.
\end{equation}
Lemma~\ref{lem:critical-holomorphy} below proves that the last function is
integrable.  Dominated convergence in \eqref{eq:polar-J} gives
\eqref{eq:residue-angular-raw}.

We verify Essouabri's hypothesis directly.  Write
\(L_\beta(z)=\sum_i b_i(\beta)z_i\), so
\(P=\prod_{\beta\in\Psi^+}L_\beta\), and fix \(0<B_0<1\).  Because the simple
coroots occur among the factors,
\[
 P(x)\geq B_0^{N-1}\max_i x_i
 \qquad (x\in[B_0,\infty)^r),
\]
so \(P(x)\to\infty\) there as \(\|x\|\to\infty\).  Put
\(M_0=\max_{\beta\in\Psi^+}\|b(\beta)\|_2\) and
\(\delta_0=B_0/(2M_0)\).  If
\(\operatorname{dist}(z,[B_0,\infty)^r)<\delta_0\), choose
\(x\in[B_0,\infty)^r\) with \(\|z-x\|_2<\delta_0\).  Since every
\(b(\beta)\) is a nonzero vector with nonnegative integral entries,
\[
 |L_\beta(z)|
 \geq L_\beta(x)-\|b(\beta)\|_2\|z-x\|_2
 >\frac {B_0}2.
\]
Thus the complex zero set of \(P\) has positive distance from
\([B_0,\infty)^r\).  These are precisely Hypothesis \(H_0S\) in
\cite[p.~432]{Essouabri1997}.  By \cite[Theorem~2, p.~436]{Essouabri1997},
\(\xi_\Phi\) continues meromorphically to \(\C\), hence in particular to a
neighborhood of \(s_0\).
The finite positive limit of \((s-s_0)\xi_\Phi(s)\) just obtained proves
that the pole is exactly simple, not merely a real-axis divergence.
\end{proof}

\section{Evaluation of the angular integral at the critical exponent}
\label{sec:MM}

The endpoint argument applies to every irreducible finite real reflection
group, the natural scope of the Macdonald--Mehta identity.  It is an endpoint
consequence of that identity, not a new version of it.

Choose, for each positive reflection hyperplane, the normal of squared
length \(2\) that is positive on a fixed chamber, and write
\[
 \widehat\Delta(v)=\prod_{\alpha>0}(v,\alpha).
\]

\subsection{Endpoint integrability}

\begin{lemma}[Holomorphy of the angular integral at the critical exponent]
\label{lem:critical-holomorphy}
Let \(W\) be an irreducible finite real reflection group of rank \(r\),
degrees \(d_1\leq\cdots\leq d_r=h\), and
\(N=\sum_i(d_i-1)=rh/2\).  For \(t>0\), define
\(t^{2k}=\exp(2k\log t)\), and set
\[
 M(k)=\int_{\R^r}e^{-\|v\|^2/2}|\widehat\Delta(v)|^{2k}\dd v,
 \qquad
 A(k)=\int_{S^{r-1}}|\widehat\Delta(\theta)|^{2k}\dd\sigma(\theta).
\]
Then \(M(k)\) converges absolutely exactly for \(\Re k>-1/h\) and is
holomorphic there.  There is \(\eta>0\) such that \(A(k)\) converges
absolutely and is holomorphic for \(\Re k>-1/h-\eta\).  In particular,
\(A(-1/h)\) is a convergent integral and
\(A(k)\to A(-1/h)\) as \(k\to-1/h\) within this domain.
\end{lemma}

\begin{proof}
At a point \(\theta\in S^{r-1}\), the
reflection hyperplanes containing \(\theta\) form the arrangement of its
parabolic stabilizer; the stabilizer is generated by these reflections
\cite[Theorem~1.1]{Lehrer2004}.  In a normal slice to the corresponding
flat, the vanishing part of \(\widehat\Delta\) is, up to a smooth
nonvanishing factor, the product of the discriminants of the irreducible
parabolic components.  This follows directly by separating the linear
forms that vanish on the flat from those that do not.

For a crystallographic component \(Q\), every proper irreducible parabolic
component has Coxeter number \(h_Q<h\)
\cite[Lemma~2.7]{HasaStasinski2019}.  The same assertion for the
noncrystallographic groups is immediate from their proper parabolic types:
\(I_2(m)\) has only \(A_1\); \(H_3\) has types among
\(A_1,A_1^2,A_2,I_2(5)\); and \(H_4\) has, in addition, types contained in
\(H_3\) and \(A_3\).  Their Coxeter numbers are strictly below
\(m,10,30\), respectively; see the tables in \cite{Humphreys1990}.

Induct on rank.  For an irreducible component \(Q\) of rank \(q\) and a real exponent
\(u\), polar coordinates at the origin leave the radial exponent
\[
 q-1+2u|Q^+|=q-1+uqh_Q,
 \]
so \(|\Delta_Q|^{2u}\) is locally integrable at the origin exactly when
\(u>-1/h_Q\).  At a spherical stratum of the ambient group, every
component is proper, hence the exponent \(u=-1/h\) lies strictly inside
all these local ranges.  Reducible stabilizers are handled componentwise by
Fubini.  Compactness of the sphere and the finiteness of parabolic types
give a uniform \(\eta>0\).  In rank one there are no spherical
singularities, so choose any \(\eta>0\).

The angular statement follows.  Polar coordinates show that the Gaussian
integral has radial factor
\(\int_0^1\rho^{r-1+2uN}\dd\rho\), which converges exactly for
\(u>-1/h\); a spherical patch avoiding all hyperplanes proves divergence
in the complementary range.  On compact parameter sets, split
\(t=|\widehat\Delta|\) into \(t\leq1\) and \(t>1\).  The smallest and
largest real parts of \(k\) give integrable majorants, while the Gaussian
controls infinity.  Morera's theorem and Fubini yield holomorphy of both
parameter integrals.  Finally, if
\(-1/h\leq\Re k\leq-1/h+\varepsilon\), then
\[
 |\widehat\Delta(\theta)|^{2\Re k}
 \leq\max\{1,L^{2\varepsilon}\}
      |\widehat\Delta(\theta)|^{-2/h},
 \quad L=\max_{S^{r-1}}|\widehat\Delta|,
\]
which proves endpoint continuity by dominated convergence.
\end{proof}

The flat criterion in \cite[Corollary~4.4]{KostaWindisch2026} gives an
alternative proof of local integrability; the argument above also records
the strictness needed here.

\subsection{Evaluation of the critical spherical integral}

\begin{proposition}[Evaluation of the spherical integral]
\label{prop:critical-sphere}
With the notation of Lemma~\ref{lem:critical-holomorphy},
\begin{equation}\label{eq:sphere-value}
  \int_{S^{r-1}}|\widehat\Delta(\theta)|^{-2/h}\dd\sigma(\theta)
 =r(2\pi)^{r/2}
 \frac{\displaystyle\prod_{i=1}^{r-1}
       \Gamma\!\left(1-\frac{d_i}{h}\right)}
      {\displaystyle\Gamma\!\left(1-\frac1h\right)^r}.
\end{equation}
This holds for crystallographic and noncrystallographic irreducible finite
real reflection groups, including \(H_3,H_4\), and \(I_2(m)\).
\end{proposition}

\begin{proof}
In Etingof's normalization, the equal-parameter Macdonald--Mehta identity
is \cite[Theorem~3.1]{Etingof2010}
\begin{equation}\label{eq:MM}
 M(k)=(2\pi)^{r/2}\prod_{i=1}^r
       \frac{\Gamma(1+kd_i)}{\Gamma(1+k)}.
\end{equation}
It is initially proved in a nonnegative parameter range.  By
Lemma~\ref{lem:critical-holomorphy}, the integral is holomorphic on
\(\Re k>-1/h\).  Every gamma argument in \eqref{eq:MM} has positive real
part there, so the product is holomorphic as well.  The identity theorem
extends \eqref{eq:MM} throughout this half-plane.

Radial separation gives
\begin{equation}\label{eq:MM-radial}
 M(k)=A(k)2^{z-1}\Gamma(z),
 \qquad z=\frac r2+kN=N\left(k+\frac1h\right).
\end{equation}
Combining \eqref{eq:MM} and \eqref{eq:MM-radial}, for
\(\Re k>-1/h\) we have
\begin{equation}\label{eq:angular-gamma-quotient}
 A(k)=2^{1-z}(2\pi)^{r/2}
 \frac{\displaystyle\prod_{i=1}^r\Gamma(1+kd_i)}
      {\displaystyle\Gamma(1+k)^r\Gamma(z)}.
\end{equation}
Set \(\varepsilon=k+1/h\), so \(z=N\varepsilon\).  Since the top degree
\(d_r=h\) has multiplicity one, the only apparently singular quotient in
\eqref{eq:angular-gamma-quotient} is
\[
 \frac{\Gamma(1+kh)}{\Gamma(z)}
 =\frac{\Gamma(h\varepsilon)}{\Gamma(N\varepsilon)}
 \longrightarrow\frac Nh.
\]
Lemma~\ref{lem:critical-holomorphy} identifies the limit with the convergent
integral \(A(-1/h)\).  Letting \(k\to-1/h\) in
\eqref{eq:angular-gamma-quotient} and using \(2N/h=r\) proves
\eqref{eq:sphere-value}.
\end{proof}

\section{Euclidean normalization and proof of Theorem~\ref{thm:main}}
\label{sec:metric}

For each positive coroot, set
\[
 \widehat\beta=\frac{\sqrt2\,\beta}{\|\beta\|},
 \qquad
 c=\prod_{\beta\in\Psi^+}\frac{\|\beta\|}{\sqrt2}.
\]
Then \(\Delta_\Psi=c\widehat\Delta\).  The absolute normalized
discriminant is \(W\)-invariant, and the sphere is the union, up to null
walls, of \(|W|\) congruent chambers.  Proposition~\ref{prop:transfer} and
Proposition~\ref{prop:critical-sphere} therefore give
\begin{equation}\label{eq:metric-formula}
 \Res_{s=2/h}\xi_\Phi(s)
 =\frac{2(2\pi)^{r/2}}{hD|W|}\,c^{-2/h}
 \frac{\displaystyle\prod_{i=1}^{r-1}
       \Gamma\!\left(1-\frac{d_i}{h}\right)}
      {\displaystyle\Gamma\!\left(1-\frac1h\right)^r}.
\end{equation}

Put \(\ell_i=\|\beta_i\|/\sqrt2\) and
\(L=\operatorname{diag}(\ell_1,\ldots,\ell_r)\).  With the normalization
of \(C_\Psi\) fixed above, \(C_\Psi=L^{-2}G\), hence
\begin{equation}\label{eq:Gram-Cartan}
 \det G=(\det C_\Psi)\prod_{i=1}^r\ell_i^2
       =(\det C_\Phi)\prod_{i=1}^r\ell_i^2.
\end{equation}
For each root length, the number of positive coroots of that length is
\(h/2\) times the number of simple coroots of that length.  The only
non-simply-laced cases are displayed here; ``long/short'' refers to \(\Psi\).
\begin{center}
\begin{tabular}{@{}ccll@{}}
\toprule
\(\Psi\) & \(h/2\) & simple long/short & positive long/short\\
\midrule
\(B_r\), \(r\geq2\) & \(r\) & \(r-1,1\) & \(r(r-1),r\)\\
\(C_r\), \(r\geq2\) & \(r\) & \(1,r-1\) & \(r,r(r-1)\)\\
\(F_4\) & \(6\) & \(2,2\) & \(12,12\)\\
\(G_2\) & \(3\) & \(1,1\) & \(3,3\)\\
\bottomrule
\end{tabular}
\end{center}
The simply laced case is immediate, and the table proves in every case that
\begin{equation}\label{eq:c-simple}
 c^{2/h}=\prod_{i=1}^r\ell_i.
\end{equation}
Since \(D=(\det G)^{-1/2}\), equations \eqref{eq:Gram-Cartan} and
\eqref{eq:c-simple} imply
\begin{equation}\label{eq:metric-cancel}
 \frac{c^{-2/h}}D=\sqrt{\det C_\Phi}.
\end{equation}
Substitution in \eqref{eq:metric-formula} proves \eqref{eq:main};
\eqref{eq:ordinary-residue} follows from \eqref{eq:normalization} and
\eqref{eq:K-degrees}.

\section{Consequences, extensions, and examples}
\label{sec:consequences}

\subsection{A direct irreducible-representation counting law}

\begin{corollary}\label{cor:representation-count}
Let
\[
 \mathcal N_\Phi(X)=\#\{\lambda\in P_+:\dim V_\lambda\leq X\}.
\]
Then, as \(X\to\infty\),
\begin{equation}\label{eq:representation-count}
 \mathcal N_\Phi(X)\sim \mathcal C_\Phi X^{2/h},
\end{equation}
where
\begin{equation}\label{eq:counting-constant}
  \mathcal C_\Phi=
 \frac{(2\pi)^{r/2}\sqrt{\det C_\Phi}}{|W|}
 \left(\prod_{i=1}^r(d_i-1)!\right)^{2/h}
 \frac{\displaystyle\prod_{i=1}^{r-1}
       \Gamma\!\left(1-\frac{d_i}{h}\right)}
      {\displaystyle\Gamma\!\left(1-\frac1h\right)^r}.
\end{equation}
Equivalently,
\(\mathcal C_\Phi=(h/2)\Res_{s=2/h}\zeta_\Phi(s)\).
\end{corollary}

\begin{proof}
Put
\[
 \Omega=\{x\in(0,\infty)^r:P(x)\leq1\},\qquad V=\vol_r(\Omega).
\]
The change of variables and angular integral already used above give
\begin{equation}\label{eq:star-volume}
 V=\frac1{rD}\int_{\mathcal S}
       \Delta_\Psi(\theta)^{-s_0}\dd\sigma(\theta)
  =\frac1{s_0}\Res_{s=s_0}\xi_\Phi(s)<\infty.
\end{equation}
For \(T>0\), define
\[
 H(T)=\#\{m\in\N^r:P(m)\leq T^N\}.
\]
For every counted \(m\), the cube
\(\prod_i(m_i-1,m_i]\) lies in \(T\Omega\), up to coordinate
hyperplanes, because \(P\) is coordinatewise nondecreasing.  Thus
\begin{equation}\label{eq:star-upper}
 H(T)\leq VT^r.
\end{equation}

For \(\varepsilon>0\), let
\(\Omega_\varepsilon=\Omega\cap[\varepsilon,\infty)^r\) and
\(V_\varepsilon=\vol_r(\Omega_\varepsilon)\).  The simple-root factors
and the remaining positive forms show
\(P(x)\geq\varepsilon^{N-1}\max_i x_i\) on this orthant, so
\(\Omega_\varepsilon\) is bounded.  Also
\(V_\varepsilon\uparrow V\) as \(\varepsilon\downarrow0\).  If
\(T>\varepsilon^{-1}\), put \(T_\varepsilon=T-\varepsilon^{-1}\).
For \(x\in T_\varepsilon\Omega_\varepsilon\), let
\(m_i=\lceil x_i\rceil\).  Then
\[
 m_i\leq x_i+1\leq\frac{T}{T_\varepsilon}x_i,
 \qquad
 P(m)\leq\left(\frac{T}{T_\varepsilon}\right)^NP(x)\leq T^N.
\]
The associated unit cubes cover \(T_\varepsilon\Omega_\varepsilon\),
so
\begin{equation}\label{eq:star-lower}
 H(T)\geq V_\varepsilon T_\varepsilon^r.
\end{equation}
Let \(T\to\infty\) and then \(\varepsilon\downarrow0\) in
\eqref{eq:star-upper}--\eqref{eq:star-lower}.  This proves
\(H(T)\sim VT^r\), without a Tauberian theorem.  Since
\(\dim V_\lambda=P(m)/K_\Phi\), take
\(T=(K_\Phi X)^{1/N}\) and use \eqref{eq:star-volume} and
\eqref{eq:ordinary-residue}.  Formula \eqref{eq:counting-constant} then
follows from Theorem~\ref{thm:main}.
\end{proof}

Romik used the \(A_2\) Witten zeta function to obtain a Hardy--Ramanujan-type
asymptotic for the number of all \(n\)-dimensional representations of
\(SU(3)\); that representation-partition problem is distinct from the
cumulative count of irreducible representations in
Corollary~\ref{cor:representation-count} \cite{Romik2017}.

\subsection{Au's conjecture and duality}

\begin{corollary}\label{cor:au}
Theorem~\ref{thm:main} establishes Conjecture~1.3 of Au
\cite{Au2024} for every irreducible crystallographic root system.
\end{corollary}

\begin{proof}
Au's conjectured class is \(\Qbar\) times products of gamma values at
rational arguments.  Every argument in \eqref{eq:main} is rational;
\(2^{r/2}\sqrt{\det C_\Phi}/(h|W|)\) is algebraic; and
\(\pi^{r/2}=\Gamma(1/2)^r\).  Thus all powers of \(\pi\), including
half-integral ones, are gamma values times an algebraic factor.  The
denominator contributes negative integral powers of the permitted gamma
values.  Hence \eqref{eq:main} lies in precisely the predicted class.
\end{proof}

In the rank-two and rank-three cases treated in Au's paper, his
Mellin-transform formulas first express the leading residue as a positive
cone integral, which becomes the corresponding chamber integral after radial
decomposition; Proposition~\ref{prop:transfer} gives the same identification
uniformly.

\begin{corollary}[Duality]\label{cor:duality}
The normalized and ordinary leading residues, as well as the counting
constant, are invariant under \(\Phi\leftrightarrow\Phi^\vee\).
\end{corollary}

\begin{proof}
The degrees, \(h\), \(|W|\), and \(K_\Phi\) are invariant under duality,
and the two Cartan matrices are transposes.
\end{proof}

\subsection{Root-system data and explicit residues}

\begin{center}
\small
\begin{tabular}{@{}cclrc@{}}
\toprule
Type & \(h\) & Degrees & \(|W|\) & \(\det C_\Phi\)\\
\midrule
\(A_r\ (r\geq1)\) & \(r+1\) & \(2,3,\ldots,r+1\) & \((r+1)!\) & \(r+1\)\\
\(B_r,C_r\ (r\geq2)\) & \(2r\) & \(2,4,\ldots,2r\) & \(2^r r!\) & \(2\)\\
\(D_r\ (r\geq4)\) & \(2r-2\) & \(2,4,\ldots,2r-2,r\) & \(2^{r-1}r!\) & \(4\)\\
\(G_2\) & 6 & \(2,6\) & 12 & 1\\
\(F_4\) & 12 & \(2,6,8,12\) & 1152 & 1\\
\(E_6\) & 12 & \(2,5,6,8,9,12\) & 51840 & 3\\
\(E_7\) & 18 & \(2,6,8,10,12,14,18\) & 2903040 & 2\\
\(E_8\) & 30 & \(2,8,12,14,18,20,24,30\) & 696729600 & 1\\
\bottomrule
\end{tabular}
\end{center}
For \(D_r\), the displayed degrees are a multiset; at \(r=4\) they are
\(2,4,4,6\), and the top degree remains unique.

Write
\[
 R_\Phi=\Res_{s=2/h}\xi_\Phi(s).
\]

\begin{corollary}[Classical families]\label{cor:classical-closed}
For \(h=r+1\),
\begin{equation}\label{eq:Ar}
 \boxed{
 R_{A_r}=
 \frac{\bigl[2\sin(\pi/h)\Gamma(1/h)\bigr]^h}
      {\pi h h!}},\qquad r\geq1.
\end{equation}
Thus the trigonometric factor in type \(A\) is exactly
\((2\sin(\pi/h))^h\), and \(R_{A_1}=1\).
For \(r\geq2\),
\begin{equation}\label{eq:BCr}
 \boxed{
 R_{B_r}=R_{C_r}=
 \frac{\sin^r(\pi/(2r))}{\sqrt\pi\,r^{3/2}r!}
 \Gamma\!\left(\frac1{2r}\right)^r.}
\end{equation}
The equality is forced because \(B_r\) and \(C_r\) have the same rank,
Coxeter number, Weyl-group order, Cartan determinant, and degree multiset,
which are all the data entering Theorem~\ref{thm:main}.  For \(r\geq4\),
\begin{equation}\label{eq:Dr}
 \boxed{
 R_{D_r}=
 \frac{2\sin^r(\pi/(2r-2))}{\pi(r-1)^{3/2}r!}
 \Gamma\!\left(\frac{r-2}{2r-2}\right)
 \Gamma\!\left(\frac1{2r-2}\right)^r.}
\end{equation}
At \(r=4\), this reduces further to
\begin{equation}\label{eq:D4-collapse}
 \boxed{
 R_{D_4}=\frac{\Gamma(1/3)^9}{2^{22/3}\sqrt3\,\pi^3}.}
\end{equation}
\end{corollary}

\begin{proof}
For type \(A\), Gauss multiplication at \(1/h\) evaluates the complete
product \(\prod_{j=1}^{h-1}\Gamma(j/h)\), and reflection at \(1/h\)
gives \eqref{eq:Ar}.  For types \(B\) and \(C\), the numerator arguments
form the complete set \(1/r,\ldots,(r-1)/r\); Gauss multiplication and
reflection at \(1/(2r)\) give \eqref{eq:BCr}.  For \(D_r\), the even
degrees give the complete Gauss product at level \(r-1\), while the
additional degree \(r\) contributes
\(\Gamma((r-2)/(2r-2))\).  Reflection gives \eqref{eq:Dr}; duplication at
\(1/6\) and reflection at \(1/3\) give \eqref{eq:D4-collapse}.
\end{proof}

Put
\[
 S_6=\sin\frac\pi4\sin\frac\pi3\sin\frac{5\pi}{12}
     =\frac{3+\sqrt3}{8},
 \qquad
 S_8=\sin\frac{4\pi}{15}\sin\frac\pi3
      \sin\frac{2\pi}{5}\sin\frac{7\pi}{15}.
\]

\begin{corollary}[Exceptional types]\label{cor:exceptional-closed}
The exceptional leading residues have the following exact normal forms.
The decimal column gives twenty significant digits.

\begin{center}
\footnotesize
\renewcommand{\arraystretch}{1.55}
\begin{tabularx}{\textwidth}{@{}c>{\raggedright\arraybackslash}Xr@{}}
\toprule
Type & Exact normal form & Decimal\\
\midrule
\(G_2\) &
\(\displaystyle \frac{\Gamma(1/3)^3}{2^{8/3}3^{3/2}\pi}\) &
\(1.8548544575371756409\!\times\!10^{-1}\)\\
\(F_4\) &
\(\displaystyle \frac{2^{4/3}\sin^4(\pi/12)}{1728\sqrt3}
 \frac{\Gamma(1/12)^4}{\Gamma(1/3)}\) &
\(2.4660172867424241539\!\times\!10^{-2}\)\\
\(E_6\) &
\(\displaystyle \frac{2^{7/12}3^{3/8}\sin^6(\pi/12)}
 {38880\sqrt\pi\sqrt{S_6}}\Gamma(1/12)^6\) &
\(2.9667111143742969785\!\times\!10^{-2}\)\\
\(E_7\) &
\(\displaystyle \frac{2^{2/9}\sin^7(\pi/18)}
 {1632960\sin(2\pi/9)\sin(\pi/3)\sin(4\pi/9)}
 \frac{\Gamma(1/18)^7}{\Gamma(7/18)}\) &
\(1.3522189960347867977\!\times\!10^{-3}\)\\
\(E_8\) &
\(\displaystyle \frac{\sin^8(\pi/30)}
 {326592000\,3^{1/5}5^{1/12}\sin(\pi/5)\sqrt{S_8}}
 \frac{\Gamma(1/30)^8}{\Gamma(4/15)}\) &
\(1.1174719544509934428\!\times\!10^{-5}\)\\
\bottomrule
\end{tabularx}
\end{center}

After rational factors and powers of \(\pi\) are separated, let
\(a_6,a_7,a_8\) denote the displayed algebraic coefficients for
\(E_6,E_7,E_8\), respectively.  Their minimal coefficient fields are
\[
 \mathbb Q(\sqrt[3]{2},\sqrt3)\quad(G_2,F_4),
 \qquad [\mathbb Q(\sqrt3,a_6):\mathbb Q]=24,
\]
\[
 [\mathbb Q(2\cos(2\pi/9),a_7):\mathbb Q]=54,
 \qquad
 [\mathbb Q(\zeta_{30})^+(a_8):\mathbb Q]=120.
\]
In particular, the square-root sine products in the \(E_6\) and \(E_8\)
coefficients do not disappear inside the smaller real cyclotomic fields.
\end{corollary}

\begin{proof}
Insert the exceptional degree data into Theorem~\ref{thm:main}.  Euler
reflection, duplication, and Gauss multiplication reduce the resulting
gamma monomials.  For \(E_6\), the product of the reflection identities at
\(1/4,1/3,5/12\), and \(1/2\) yields the positive square root
\(S_6^{-1/2}\).  For \(E_8\), the multiplication formulas at levels three
and five, followed by reflection at
\(1/5,4/15,1/3,2/5\), and \(7/15\), yield \(S_8^{-1/2}\).
The exact norm and discriminant calculations are recorded in the accompanying
file \texttt{coefficient\_field\_certificates.txt}; no numerical relation
search is used.
\end{proof}

\subsection{The \texorpdfstring{\(A_4\)}{A4} identity by hand}

Put \(\Gamma_j=\Gamma(j/5)\).  Formula \eqref{eq:Ar} gives
\[
 R_{A_4}=\frac{\Gamma_1^2\Gamma_2^2\Gamma_3^2}
 {60\Gamma_4^3}.
\]
Euler's reflection formula yields
\[
 \Gamma_1\Gamma_4=\frac\pi{\sin(\pi/5)},
 \qquad
 \Gamma_2\Gamma_3=\frac\pi{\sin(2\pi/5)}.
\]
Consequently,
\begin{equation}\label{eq:A4-trig}
 R_{A_4}=\frac{\Gamma_1^5}{60\pi}
 \frac{\sin^3(\pi/5)}{\sin^2(2\pi/5)}.
\end{equation}
If \(s=\sin(\pi/5)\) and \(c=\cos(\pi/5)\), then the trigonometric
factor is \(s/(4c^2)\).  Using
\[
 s^2=\frac{5-\sqrt5}{8},\qquad
 c^2=\frac{3+\sqrt5}{8},
\]
one finds
\[
 \left(\frac{s}{4c^2}\right)^2
 =\frac{50-22\sqrt5}{16}.
\]
All quantities are positive; substitution in \eqref{eq:A4-trig} gives
\begin{equation}\label{eq:A4-final}
  R_{A_4}=\frac{\sqrt{50-22\sqrt5}}{240\pi}\Gamma(1/5)^5,
\end{equation}
which is Au's prediction.  Moreover,
\(K_{A_4}=1!2!3!4!=288\), so the ordinary residue is
\(288^{2/5}R_{A_4}\).

\section{Reduction of the gamma quotient}
\label{sec:gamma-reductions}

This section records exactly what can be reduced using Euler reflection and
Gauss multiplication.  Let \(\mathcal U^-_{\mathbb Q}\) be the
\(\mathbb Q\)-vector space on symbols \([x]\), \(x\in\mathbb Q/\mathbb Z\),
modulo
\begin{equation}\label{eq:distribution-relations}
 [x]+[-x]=0,
 \qquad
 [x]=\sum_{k=0}^{q-1}\left[\frac{x+k}{q}\right]\quad(q\geq2).
\end{equation}
Algebraic factors and powers of \(\pi\) are suppressed in this quotient.
For the gamma quotient in Theorem~\ref{thm:main}, put
\begin{equation}\label{eq:degree-class}
 v_\Phi=
 \sum_{i=1}^{r-1}\left[\frac{h-d_i}{h}\right]
 -r\left[\frac{h-1}{h}\right],
\end{equation}
and let \(\mu(\Phi)\) be the least number of distinct symbols in a
representative of \(v_\Phi\).  Thus every minimality statement below is
only modulo the two relation families in \eqref{eq:distribution-relations};
no completeness assertion for all algebraic relations among rational gamma
values is made or needed.

\begin{proposition}[The invariant in terms of exponents]\label{prop:exponent-count}
Let \(m_i=d_i-1\) be the exponent multiset, in increasing order.  Then
\begin{equation}\label{eq:exponent-class}
 \boxed{
 v_\Phi=
 \sum_{j=1}^{r}\left[\frac{m_j-1}{h}\right]
 -r\left[\frac{h-1}{h}\right].}
\end{equation}
For an integer \(x\), define
\[
 C_h(x)=
 \begin{cases}
 0,&h\mid x,\\
 2(x\bmod h)-h,&h\nmid x,
 \end{cases}
\]
and, for a positive unit \(t<h/2\), put
\(N_\Phi(t)=\#\{j:m_j<t\}\).  Define the linear functional
\(\widehat\beta_t\) by
\(\widehat\beta_t([a/h])=C_h(ta)\).  It respects the relations in
\eqref{eq:distribution-relations}, and satisfies
\begin{equation}\label{eq:signed-count}
 \boxed{
 \widehat\beta_t(v_\Phi)=h\bigl(2N_\Phi(t)-r+1\bigr).}
\end{equation}
\end{proposition}

\begin{proof}
The degree duality \(d_i+d_{r+1-i}=h+2\), which is immediate from the
listed degree multisets, gives
\(h-d_i=m_{r+1-i}-1\).  Reindexing \eqref{eq:degree-class} and adding the
zero term \([(m_1-1)/h]=[0]\) proves \eqref{eq:exponent-class}.

The Coxeter characteristic polynomial has integer coefficients, so
multiplication by a unit modulo \(h\) permutes the exponent multiset.  The
exponent \(1\) occurs once; hence every unit residue occurs once.  If \(b\)
is the representative of \(tm\) in \(\{1,\ldots,h-1\}\), then
\(\sum b=rh/2\).  For \(b\ne t\),
\[
 C_h(b-t)=2(b-t)-h+2h\mathbf 1_{b<t},
\]
while the unique term \(b=t\) contributes the correction \(h\).  Summing
and subtracting the denominator contribution
\(rC_h(t(h-1))=r(h-2t)\) gives \eqref{eq:signed-count}.
The identity
\(\sum_{k=0}^{q-1}C_{qh}(x+kh)=qC_h(x)\) verifies directly that this
functional respects \eqref{eq:distribution-relations}.
\end{proof}

\begin{lemma}[Denominator of a representative with one nonzero symbol]
\label{lem:denominator-confinement}
Let \(N\) be even and let a nonzero class in the span of
\([a/N]\) have a one-symbol representative \(c[x]\), with \(c\ne0\).
Then the order of \(x\) divides \(N\).
\end{lemma}

\begin{proof}
Use the odd Bernoulli function
\[
 \widetilde B_1(y)=
 \begin{cases}
 \{y\}-\tfrac12,&y\ne0,\\
 0,&y=0.
 \end{cases}
\]
It satisfies the distribution relations.  Every profinite unit congruent
to \(1\pmod N\) fixes the original level-\(N\) class and therefore gives
\(\widetilde B_1(ux)=\widetilde B_1(x)\).  A nonzero odd symbol excludes
\(x=0,1/2\), and \(\widetilde B_1\) is injective on the remaining
representatives in \([0,1)\); hence \(ux=x\).  If the order of \(x\)
contained a prime power larger than its counterpart in \(N\), one could
choose such a unit locally at that prime but nontrivially modulo the larger
power, a contradiction.
\end{proof}

\begin{theorem}[Criterion for reduction to a single gamma argument]\label{thm:one-gamma-criterion}
For every irreducible crystallographic \(\Phi\ne A_1\),
\begin{equation}\label{eq:collapse-criterion}
 \boxed{
 \mu(\Phi)=1
 \quad\Longleftrightarrow\quad
 N_\Phi(t)=\frac{(r-1)(t-1)}{h-2}
 \quad\left(1\leq t<\frac h2,\ (t,h)=1\right).}
\end{equation}
When this condition holds,
\begin{equation}\label{eq:collapse-ray}
 \boxed{
 v_\Phi=\frac{h(r-1)}{h-2}[1/h].}
\end{equation}
Consequently,
\[
 \mu(\Phi)=1
 \quad\Longleftrightarrow\quad
 \Phi=A_r\ (r\geq2),\ B_r,\ C_r,\ D_4,\ G_2,\ E_6,
\]
whereas
\[
 \mu(\Phi)=2
 \quad\Longleftrightarrow\quad
 \Phi=D_r\ (r\geq5),\ F_4,\ E_7,\ E_8,
 \qquad \mu(A_1)=0.
\]
\end{theorem}

\begin{proof}
Order the positive unit classes as
\(1=t_1<\cdots<t_s<h/2\).  Each \(t_i\) is an exponent, so
\(N_\Phi(t_i)\) increases strictly; exponent duality gives
\(r\geq2N_\Phi(t_i)+2\).  Hence the values in
\eqref{eq:signed-count} are strictly increasing and negative.

Suppose first that \(\mu(\Phi)=1\).  For type \(A_r\), \(h=r+1\), so
\((r-1)/(h-2)=1\), and the exponents \(1,\ldots,r\) give
\(N_\Phi(t)=t-1\).  For types \(B_r\) and \(C_r\), \(h=2r\), so
\((r-1)/(h-2)=1/2\); every unit modulo \(2r\) is odd, and the exponents
\(1,3,\ldots,2r-1\) give \(N_\Phi(t)=(t-1)/2\).  Every remaining family
has even \(h\).  Lemma~\ref{lem:denominator-confinement} reduces a
representative to \(c[a/h]\).  Reflection permits \(c>0\).  The column
\(C_h(at_i)\) must then be strictly increasing and negative.  This forces
\(a\equiv\pm1\pmod h\): if \(g=(a,h)\) and \(H=h/g\), strictness gives an
injection from unit classes modulo sign at level \(h\) to those at level
\(H\), so \(\varphi(h)\leq\varphi(H)\).  Since \(H\mid h\), either
\(H=h\), when multiplication by \(a\) is an order-preserving permutation
and hence the identity, or \(h=2H\) with \(H\) odd.  In the latter case
\(a=2b\), and negativity at the first and last positive units contradicts
\(H\geq5\).  The remaining even levels \(h=4,6\) are direct.
Evaluation at \(t=1\) now gives
\(c=h(r-1)/(h-2)\), and comparison at general \(t\) gives
\eqref{eq:collapse-criterion} and \eqref{eq:collapse-ray}.

For the converse, the shifted contribution of a full order-\(d\) Galois
orbit at level \(h\) is, by M\"obius inversion and
\eqref{eq:distribution-relations},
\begin{equation}\label{eq:orbit-shift}
 \Delta_{h,d}=\varphi(d)[1/h]
 -\sum_{q\mid d}\mu_{\mathrm M}(d/q)[q/h],
\end{equation}
where \(\mu_{\mathrm M}\) is the M\"obius function.  Applying this identity
to the Coxeter spectra gives
\begin{equation}\label{eq:class-table}
\begin{array}{c|c}
\Phi&v_\Phi\\ \hline
A_r&h[1/h]\\
B_r,C_r&r[1/(2r)]\\
D_r&r[1/(2r-2)]+[(r-2)/(2r-2)]\\
G_2&3[1/3]\\
F_4&4[1/12]-[1/3]\\
E_6&6[1/12]\\
E_7&7[1/18]-[7/18]\\
E_8&8[1/30]-[4/15].
\end{array}
\end{equation}
Thus every pair satisfying \eqref{eq:collapse-criterion} has the
representative \eqref{eq:collapse-ray}.  For type \(D\), the
condition holds only at \(D_4\); among the exceptional types it holds for
\(G_2\) and \(E_6\).  The necessity already proved excludes a
representative with one nonzero symbol in every remaining case; the table
gives representatives with two nonzero symbols, and
\eqref{eq:signed-count} shows the classes are nonzero.  This proves the
classification.
\end{proof}

\begin{corollary}\label{cor:two-gamma-bound}
For every irreducible crystallographic root system,
\[
 \boxed{\mu(\Phi)\leq2.}
\]
\end{corollary}

\begin{proof}
The representatives in \eqref{eq:class-table} have at most two symbols.
The proof is uniform after the Coxeter spectra are decomposed into their
cyclotomic orbits by \eqref{eq:orbit-shift}; exponent duality alone is not
being used as a substitute for that decomposition.
\end{proof}

\begin{remark}[The contrast between \(E_6\) and \(F_4\)]
Both types have \(h=12\).  Their exponent multisets are
\[
 M_{F_4}=\{1,5,7,11\},
 \qquad
 M_{E_6}=\{1,5,7,11\}\sqcup\{4,8\}.
\]
At the decisive unit \(t=5\),
\[
 N_{F_4}(5)=1\ne\frac{3(5-1)}{10}=\frac65,
 \qquad
 N_{E_6}(5)=2=\frac{5(5-1)}{10}.
\]
The additional order-three Coxeter orbit in \(E_6\) supplies exactly the
missing count.  Thus the opposite conclusions at the same Coxeter number
are explained by the exponents after subtracting one, not by a numerical
search.
\end{remark}

\appendix

\section{Independent numerical verification of the chamber integral}
\label{app:quadrature}

Let
\[
 \Delta^{r-1}=\{u_i\geq0:\ u_1+\cdots+u_r=1\},
 \qquad
 P_\Phi(u)=\prod_{\beta\in\Psi^+}L_\beta(u),
\]
and define
\begin{equation}\label{eq:chamber-period}
 P_{\mathrm{ch}}(\Phi)=\frac2r
 \int_{\Delta^{r-1}}P_\Phi(u)^{-2/h}\,\dd u.
\end{equation}
We derive the identity used in the numerical comparison from the preceding
sections.  In the positive cone write \(x=tu\), where
\(t=x_1+\cdots+x_r\) and \(u\in\Delta^{r-1}\).  Homogeneity gives
\[
 P_\Phi(tu)^{-s}\,\dd x
 =t^{r-1-Ns}P_\Phi(u)^{-s}\,\dd t\,\dd u.
\]
The region \([1,\infty)^r\) corresponds to
\(t\geq T(u):=\max_i u_i^{-1}\).  Hence, initially for
\(\Re s>r/N\),
\[
 J(s)=\int_{\Delta^{r-1}}
 \frac{T(u)^{r-Ns}}{Ns-r}P_\Phi(u)^{-s}\,\dd u.
\]
The boundary convergence proved in Section~\ref{sec:faces} permits passage
to \(s=2/h=r/N\), while Section~\ref{sec:transfer} identifies the lattice
and integral leading coefficients.  The displayed radial factor has residue
\(1/N\).  Since \(N=rh/2\),
\begin{equation}\label{eq:period-residue}
 P_{\mathrm{ch}}(\Phi)=hR_\Phi
 =\frac{2(2\pi)^{r/2}\sqrt{\det C_\Phi}}{|W|}
 \frac{\prod_{i=1}^{r-1}\Gamma(1-d_i/h)}
      {\Gamma(1-1/h)^r}.
\end{equation}

The independent computation integrated the elementary simplex integral in
\eqref{eq:chamber-period}, without using the gamma product in the integration
code.  Positive coroots were generated by exact integer reflection closure.
The simplex was divided into its \(r!\) Hepp order sectors.  On a sector
\(x_{\pi_0}\geq\cdots\geq x_{\pi_{r-1}}\), the substitution
\[
 v_0=1,\qquad v_j=y_1\cdots y_j,\qquad
 x_{\pi_j}=\frac{v_j}{\sum_kv_k}
\]
cancels the projective denominator exactly because \((2/h)N=r\).  The
remaining integrand is a monomial Jacobi weight times a function that is
separately completely monotone.  Tensor products of Gauss--Jacobi rules
therefore give lower bounds, while left-endpoint Gauss--Radau--Jacobi rules
give upper bounds.

The following table reports the closed value from \eqref{eq:period-residue},
the half-width of the independent enclosure, and the number of common
significant digits after the allowance for arithmetic error.  In every row
the closed value lies inside the enclosure.

\begin{center}
\scriptsize
\renewcommand{\arraystretch}{1.2}
\begin{tabular}{@{}crrc@{}}
\toprule
Type & \(P_{\mathrm{ch}}(\Phi)\) & Enclosure half-width & Common digits\\
\midrule
\(A_2\) & 5.2999162508563498719 & \(8.54\times10^{-50}\) & 49\\
\(A_3\) & 9.1669144240271637700 & \(6.65\times10^{-39}\) & 38\\
\(A_4\) & 12.144365427721648164 & \(5.20\times10^{-28}\) & 28\\
\(B_2\) & 2.6220575542921198105 & \(4.75\times10^{-38}\) & 37\\
\(B_3\) & 2.3407593555076023087 & \(6.27\times10^{-30}\) & 29\\
\(B_4\) & 1.6242821124629148599 & \(1.36\times10^{-26}\) & 26\\
\(D_4\) & 4.9232563217488548831 & \(1.25\times10^{-28}\) & 28\\
\(G_2\) & 1.1129126745223053846 & \(1.75\times10^{-32}\) & 32\\
\(F_4\) & 0.29592207440909089847 & \(6.24\times10^{-28}\) & 26\\
\bottomrule
\end{tabular}
\end{center}

For \(E_6\), the order-two tensor rules over all \(720\) sectors gave the
valid but coarse bracket
\[
 0.33338132459886239844
 <P_{\mathrm{ch}}(E_6)<
 0.37940764457434622297.
\]
It contains the closed value
\(0.35600533372491563742\ldots\), but certifies no decimal digit.  The
direct decomposition was not carried out to a defended enclosure for
\(E_7\) or \(E_8\), which have respectively \(5040\) six-dimensional and
\(40320\) seven-dimensional sectors.  The independent corroboration is
therefore thinnest exactly where the residues are smallest and the arithmetic
is hardest: \(E_6\) is only bracketed, while \(E_7\) and \(E_8\) are not
reached.

The Gauss/Radau inequalities are analytic and one-sided.  Nodes and weights
were evaluated at 80--110 decimal working digits, checked against the exact
moments in their ranges of exactness, and enlarged by an explicit allowance
for arithmetic error.  The resulting enclosures are deterministic under that
stated
high-precision model; they are not claimed to be formally machine-verified
interval certificates.

This computation tests only the evaluation of the elementary chamber
integral in Proposition~\ref{prop:critical-sphere}.  It does not independently
test the identification of the leading coefficient in
Proposition~\ref{prop:transfer}, it does not test the holomorphy statement in
Lemma~\ref{lem:critical-holomorphy}, and it reaches the Gram/root-length
normalization only through the composite
equality \eqref{eq:period-residue}.  It is numerical evidence, not a proof
of any theorem in the paper.

The accompanying numerical enclosure records, including the \(A_4\) comparison,
document this independent evidence; they corroborate the formulas, and no proof
above depends on numerical output.

\begingroup
\footnotesize
\bibliographystyle{abbrv}
\bibliography{Witten_Residues_Macdonald_Mehta}
\endgroup

\end{document}